\documentclass[12pt,letterpaper]{article}

\newcommand{\required}[1]{\section*{\hfil \sharp1\hfil}}

\usepackage{verbatim,url}
\usepackage{graphicx}

\usepackage{amssymb,amsmath,amsfonts,amsthm}
\usepackage{hyperref}
\usepackage{upref}
\newcommand{\Beq}{\begin{equation}}
\newcommand{\Eeq}{\end{equation}}
\newcommand{\beq}{\begin{equation*}}
\newcommand{\eeq}{\end{equation*}}
\newcommand{\bal}{\begin{align}}
\newcommand{\eal}{\end{align}}

\newtheorem{theorem}{Theorem}
\newtheorem{Lemma}{Lemma}

\newtheorem{definition}{Definition}

\theoremstyle{definition}

\newtheorem{remark}{Remark}

\newcommand{\ac}[1]{\begin{quotation}\textbf{Anuj's comment:\
		}{\textit{\sharp1}}\end{quotation}}
\newcommand{\rc}[1]{\begin{quotation}\textbf{Rohit's comment:\
		}{\textit{\sharp1}}\end{quotation}}
\allowdisplaybreaks

\usepackage{xcolor}

\title{Minimax optimal estimator in the stochastic inverse problem for exponential Radon transform}
\date{}
\author{Anuj Abhishek}

\begin{document}
	\maketitle
	\begin{abstract}
	In this article, we consider the problem of inverting the exponential Radon transform of a function in the presence of noise. We propose a kernel estimator to estimate the true function, analogous to the one proposed by Korostel\"{e}v and Tsybakov in their article `Optimal rates of convergence of estimators in a probabilistic setup of tomography problem', Problems of Information Transmission, 27:73-81,1991. For the estimator proposed in this article, we then show that it converges to the true function at a minimax optimal rate.
	\end{abstract}
	\section{Introduction}
	Exponential Radon transform (ERT), which is the object of study in this article, can be thought of as a generalization of the classical Radon transform. In fact, the ERT of a compactly supported function $f(x)$ in $\mathbb{R}^2$ is given by :
	\begin{align}
	T_{\mu}f(\theta,s)=\int\limits_{x\cdot\theta=s}e^{\mu x\cdot\theta^{\perp}}f(x)dx
	\end{align}
	Here $s\in \mathbb{R}$, $\theta\in \mathrm{S}^1$ where $\mathrm{S}^1$ is the unit circle in $\mathbb{R}^2$, $\mu$ is a constant and $\theta^{\perp}$ denotes a unit vector perpendicular to $\theta$. Recall that lines in $\mathbb{R}^2$ can be parameterized as $L(\theta,s)=\{x:x\cdot\theta=s\}$. Thus, just as the classical Radon transform, ERT takes a function defined on a plane and maps it to a function defined over the set of lines parameterized by $(\theta,s)$. Such transforms arise naturally in imaging modalities such as SPECT (single photon emission computed tomography) imaging \cite{Wen_2006} and nuclear magnetic resonance imaging \cite{Louis1982}. 
	\par \noindent The exponential Radon transform is a special case of a more general transform called attenuated Radon transform which takes the integral of a function over straight lines with respect to an exponential weight that signifies a non-constant attenuation effect. We refer the readers to the article by Finch \cite{finch_io} and the textbook by Natterer and W\"{u}bbeling \cite{natt_wubb_book} for an excellent overview of the attenuated Radon transform. Indeed, the attenuated Radon transform is itself an example of a generalized Radon transform that was studied by Quinto in \cite{Quinto_80, Quinto_83}. 
	\par \noindent Inversion methods for the exponential Radon transform were derived by Natterer in \cite{Natt_79} and by Tretiak and Metz in \cite{TM_80}. Hazou and Solmon in \cite{hazou_solmon} gave filtered backprojection (FBP) type formulas for inversion of ERT using a class of filters. Such FBP type inversion formulas are based on the method of approximate inverse which were developed systematically in the articles by Louis \cite{louis_96} and Louis and Maas \cite{louis_maas}. An exhaustive treatment of the method of approximate inverse can be found in the book by Schuster \cite{Schuster_book}. Rigaud and Lakhal have used the method of approximate inverse and derived Sobolev estimates for attenuated Radon transform in \cite{Rigaud2015}, these estimates were central to proving some of the theorems in this article. Furthermore, Novikov in \cite{Novikov2002} and Natterer in \cite{Natt01} give an inversion formula for the more general attenuated Radon transform. There is extensive literature available on this subject and we give now a partial list of references where an interested reader may find important insights and advances made in the study of exponential and attenuated Radon transforms, see e.g. \cite{kuch96,Bal04,Boman04,novikov02,Monard18,Monard16,rullg04,salo_11,Shne94,Shne2}.
	\par \noindent Classical Radon transform has also been extensively studied in the stochastic setting. A detailed discussion of positron emission tomography (PET) in presence of noise can be found in the seminal article by Johnstone and Silverman \cite{JS_90}. {In \cite{Hahn_quinto}, Hahn and Quinto establish upper and lower bounds for the convergence of two probability measures in terms of the rates of convergence of their Radon transforms}. Korostel\"{e}v and Tsybakov show that optimal minimax convergence rates are attained by kernel type estimators, which are closely linked to FBP inversion methods, in \cite{Tsybakov_1991,Tsybakov_92}. An exhaustive coverage of the non-parametric estimation methods that are used to establish the optimal convergence rates in this article and elsewhere can be founds in the books written by Korostel\"{e}v and Tsybakov \cite{minimax_book} and Tsybakov \cite{Tsybakov_book}. Cavalier obtained results on efficient estimation of density in the non-parametric setting for stochastic PET problem in \cite{Cavalier_98,Cavalier_00}. In addition to the non-parametric kernel type estimators, Bayesian estimators for the stochastic problem of X-ray tomography have been studied by several authors, most notably by Lassas, Siltanen and Somersalo, see e.g. \cite{Siltanen2003,Lassas_09} and references therein. More recently, Monard, Nickl and Paternain have obtained results on efficient Bayesian inference for the attenuated X-ray transform on a Riemannian manifold, see \cite{Monard_19}.
	\par \noindent In this article, we propose a statistical kernel estimator for the ERT problem and show that it attains the optimal minimax rate of convergence. The organization of the article is as follows: in section 2, we describe the mathematical set up of the stochastic problem for ERT and recall some standard definitions from the literature. In section 3, we recall the FBP type inversion in the deterministic (noise-less) setting. In section 4, we propose a kernel type estimator and establish that it is asymptotically unbiased. Finally, in section 5 we show that this estimator attains optimal minimax rates of convergence.
	
	\section{Mathematical set-up and definitions}\label{definitions}
	In this section we will describe the mathematical framework for the problem and recall some standard definitions from the literature that will help us assess the optimality of the estimator proposed in this article. 
	\par \noindent Let $f(x):\mathbb{R}^2\to \mathbb{R}$ be a function that satisfies the following assumptions:\\
	\textbf{Assumption 1 (A1): } Let $B_1(x)=\{x:\vert\vert x\vert\vert \leq 1\}$ be the unit ball in $\mathbb{R}^2$. We assume that $f(x)$ is supported in the unit ball $B_1(x)$.\\
	\textbf{Assumption 2 (A2): }Let $\widetilde{f}(\xi)$ represent the Fourier transform of $f(x)$, i.e. $\widetilde{f}(\xi)=\int_{\mathbb{R}^2} f(x)e^{-i\xi\cdot x}dx$. We assume that the Fourier transform of $f(x)$ satisfies the following inequality,
	{$$\int_{\mathbb{R}^2}(1+\vert\vert \xi \vert\vert^2)^\beta {\vert \widetilde{f}(\xi)\vert^2} d\xi\leq L$$} for some fixed positive numbers $L$ and $\beta>1$.\\
	We will denote by $H(\beta,L)$, the class of functions satisfying assumptions A1 and A2. 
	\begin{definition}
	Let $\mathrm{S}^1$ denote the unit circle in $\mathbb{R}^2$ and $Z= \mathrm{S}^1\times [-1,1]$ be the cylinder whose points are given by $(\theta,s)$ where $s\in [-1,1]$ and $\theta \in \mathrm{S}^1$. By $\theta^{\perp}$, we will denote a unit vector perpendicular to $\theta$. The exponential Radon transform of $f\in H(\beta,L)$ is defined as the following function on $Z$:
	$$T_{\mu}f(\theta,s)=\int_{x\cdot\theta=s}e^{{\mu}x\cdot\theta^{\perp}}f(x)dx$$ where $\mu$ is a fixed constant. It is clear that if $\mu=0$, then the exponential Radon transform reduces to the case of the classical Radon transform.
	
	\end{definition}
\begin{definition}
	Associated to the exponential Radon transform, is its dual transform $$T_\mu^{\sharp}g(x)=\int_{\mathrm{S}^1}e^{\mu x\cdot\theta^\perp}g(\theta,x\cdot \theta)d\theta.$$ Clearly, for $\mu=0$, this is the backprojection operator for the classical Radon transform.
	
\end{definition}
\noindent Now we will describe the stochastic problem of exponential Radon transform. Let $\{(\theta_i,s_i)\}_{i=1}^{i=n}$ be $n$ random points on the observation space $Z$ and let the observations be of the form:
\begin{equation} \label{obs_model}
Y_i=T_{\mu}f(\theta_i,s_i)+\epsilon_i
\end{equation} 
\par \noindent We assume that the points $(\theta_i,s_i)$ are independent and identically distributed (i.i.d.) on $Z$ and $\epsilon_i$ are i.i.d. random variables with zero mean and some finite positive variance $\sigma^2$. The collection of the random points $\{(\theta_i,s_i)\}_{i=1}^{i=n}$ where observations are made is  called the design and will be denoted by $\mathcal{D}_n$. In the observation model given by equation (\ref {obs_model}), the random variables $\epsilon_i$ account for noise. The stochastic inverse problem for exponential Radon transform is to then estimate the function $f(x)$ based on the observations $Y_i$ for $i=\{1,2,\dots,n\}$. This problem is non-parametric in the sense that the function $f$ itself is not assumed to be of any parametric form but is rather assumed to belong to a general class of functions, say $\mathcal{F}$. In this article we have assumed $f\in H(\beta,L)$. Suppose one devises an estimator $\hat{f}_n(x)$ based on the observed data. One is then naturally led to ask the question, if this estimator is optimal? The most popular of such approaches to assess the optimality of estimators in a non-parametric setting is the minimax approach, which we will describe below. Let the nonparametric class of functions $\mathcal{F}$ be equipped with a semi-norm $d$. Thus the semi-distance between two elements $f\in \mathcal{F}$ and $g\in\mathcal{F}$ will be represented as $d(f,g)$ and we will use the quantity $d^2(\hat{f},f)=(d(\hat{f},f))^2$ as a measure of error between an estimator $\hat{f}$ and the true function $f$. First of all, note that as any such estimator $\hat{f}_n(x)$ will depend on the random observation points $\{(s_i,\theta_i)\}_{i=1}^{i=n}$ and  observations $\{Y_i\}_{i=1}^{i=n}$, it is better to consider the expected value of the error between the estimator and the true function (under the chosen semi-norm) as a measure of accuracy. The following definitions are standard in the literature.
\begin{definition}[\cite {JS_90,Tsybakov_book}]\label{risk_fn}
	The risk function of an estimator $\hat{f}_n(x)$ is defined as:
	$$ \mathcal{R}(\hat{f}_n,f)=E_f(d^2(\hat{f}_n,f)).$$
\end{definition}
\par 
\noindent From here on, $E_f$ will be used to denote the expectation with respect to the joint distribution of random variables $(s_i,\theta_i,Y_i)$, $i=\{1,\dots,n\}$ satisfying the model given by (\ref{obs_model}). Ideally, one would like to devise an estimator that would minimize the risk function. However, as the definition of the risk function depends on $f$ as well, one tries instead to find an overall measure of risk such as the \textit{minimax risk}.
\begin{definition}\cite[Page 78]{Tsybakov_book}\label{minimax_risk}
	Let $f(x)$  belong to some non-parametric class of functions $\mathcal{F}$. The maximum risk of an estimator $\hat{f}_n$ is defined as:
	$$r(\hat{f}_n)=\sup_{f\in\mathcal{F}} \mathcal{R}(\hat{f}_n,f). $$ Finally, the minimax risk on $\mathcal{F}$ is defined as:
	$$r_n(\mathcal{F})=\inf_{\hat{f}_n} \sup_{f\in\mathcal{F}}\mathcal{R}(\hat{f}_n,f)$$ where the infimum is taken over the set of all possible estimators $\hat{f}_n$ of $f$. Clearly, $$r_n(\mathcal{F})\leq r(\hat{f}_n).$$
\end{definition}
\begin{definition} \label{Def5} \cite[Page 78]{Tsybakov_book}\label{optimality} Let $\{\Psi_n^2\}_{n=1}^{\infty}$ be a positive sequence converging to zero. An estimator $\hat{f}_n^{*}$ is said to be minimax optimal if there exist finite positive constants $C_1$ and $C_2$ such that,
	$$ C_1\Psi_n^2\leq r_n(\mathcal{F})\leq r(\hat{f}_n^{*})\leq C_2\Psi_n^2. $$
	Furthermore, $\Psi_n^2$ is said to be the optimal rate of converegence.

\end{definition}
\par \noindent In this article, whenever we refer to the optimality of an estimator, we will mean its minimax optimality.
In section 4, we will propose an estimator for $f(x)\in H(\beta,L)$ based on the model (\ref{obs_model}) and establish its optimality in the following (semi) norms: 
\begin{enumerate}
	\item $d_1(f,g)=\vert f(x_0)-g(x_0)\vert$ \quad \quad ($x_0$ is an arbitrary fixed point in $B_1(x)$)
	\item $d_2(f,g)=\displaystyle{(\int \vert f(x)-g(x)\vert^2 dx)^{1/2}}$
	
\end{enumerate} as per definition \ref{Def5} above. We also note that the risk function defined using semi-norm $d_1$ is called the mean squared error (MSE), while the risk function defined using $d_2$ is referred to in the literature as the mean integrated squared error (MISE) of the estimator. Thus:
$$\text{MSE}(\hat{f}_n,f)=E_f(d_1^2(\hat{f}_n,f)), \quad \quad \text{MISE}(\hat{f}_n)=E_f(d_2^2(\hat{f}_n,f)) .$$
Finally, we recall the Kullback distance between two probability measures on a measurable space: 
\begin{definition} \cite[Page 84]{Tsybakov_book}
Let $P$ and $Q$ be two probability measures on some measurable space $(\mathcal{X},\mathcal{A})$. The Kullback distance between the two measures is given by,
\begin{align*}
I(P,Q)&=\int \log \frac{dP}{dQ} dP \quad \quad \text{if } P\text{ is absolutely continuous with respect to }Q\\
&=\infty \quad \quad \text{otherwise}
\end{align*}
\end{definition}

\section{An FBP reconstruction in the deterministic setting}
In this section we will describe some of the results from the deterministic set-up, i.e. when the observations as per the model given by (\ref{obs_model}) are not corrupted by noise. 
Let $\rho>0$ such that $0<\vert\mu \vert< 1/\rho$. Consider the function $K_{\rho}(\theta,s)=K_{\rho}(s)$ defined as:
\begin{equation}\label{kernel_def}
\begin{aligned}
K_{\rho}(s)&=\frac{1}{\pi}\int_{\vert\mu\vert}^{\sqrt{(1/\rho^2)+\mu^2}}r \cos(sr)dr\\
\end{aligned}
\end{equation}
 These kind of functions have been used in the context of filtered backprojection formulas for Radon transforms, see e.g. \cite[Page 237]{minimax_book}, \cite [Page 109]{Natterer_textbook}.
Let $I_{p}(t)$ denote the indicator function:
\begin{align*}
I_p(t)&=1, \quad \quad \vert t\vert < 1/ p\\
 &= 0, \quad \quad \vert t \vert \geq 1/p
\end{align*}The one dimensional Fourier transform of $K_{\rho}(\theta,s)$ (in the $s$-variable) is:
\begin{align} \label{Filter}
\widetilde{K}_{\rho}(\theta,t)
&=\vert t \vert, \quad \quad \vert\mu\vert <\vert t \vert <\sqrt{(1/\rho^2)+\mu^2}\nonumber\\
&= 0, \quad \quad \text{otherwise}.
\end{align}
\noindent In the following analysis, $\star$ will represent the operation of convolution of functions. Furthermore, whenever the convolution of two functions $f$ and $g$ defined on the cylinder $Z= \mathrm{S}^1\times \mathbb{R}$  is considered, the convolution will be understood to be taken with respect to their second variable, i.e. $$f\star g(\theta,s)=\int_{\mathbb{R}}f(\theta,s-t)g(\theta,t)dt.$$  
\begin{theorem}\cite[Page 49]{Natterer_textbook} \label{fbp_formula}
Let $f_{\rho}(x)=\frac{1}{4\pi}T^{\sharp}_{-\mu}(K_{\rho}\star T_\mu f)$. Then,	$$f(x)=\lim_{\rho\to 0}f_{\rho}(x).$$
\end{theorem}
\begin{proof}
The proof of this theorem is well known, see e.g.\cite [Section II.6] {Natterer_textbook}. However, we will reproduce it here for the sake of completeness. First of all recall that from \cite[(6.2), Page 47]{Natterer_textbook}, we know that: 
$T^{\sharp}_{-\mu}(g\star T_\mu f)=(T^{\sharp}_{-\mu}g)\star f$. Thus, if we can show that $\frac{1}{4\pi}T^{\sharp}_{-\mu}K_{\rho}$ is an approximate Dirac-delta function, then we are done. Let us then compute:
\begin{align*}\label{eqn1}
T^{\sharp}_{-\mu}K_{\rho}(x)&=\int_{\mathrm{S}^1}e^{-\mu x\cdot \theta^{\perp}}K_{\rho}(\theta,x\cdot\theta)d\theta \nonumber \\
&=\frac{1}{2\pi}\int_{\mathrm{S}^1}e^{-\mu x\cdot \theta^{\perp}}\int_{\mathbb{R}}e^{ix\cdot \theta} \widetilde{K}_{\rho}(\theta,t)dt d\theta \nonumber \\
&=\frac{1}{2\pi}\int\limits_{\vert\mu\vert<\vert t\vert <\sqrt{(1/\rho^2)+\mu^2}}\vert t \vert \int_{\mathrm{S}^1}e^{-\mu x\cdot \theta^{\perp}+i(x\cdot \theta)t}d\theta dt
\end{align*}
In what follows, by $J_0$ we will denote the Bessel function of first kind of integer order $0$. Now from \cite[VII.3.17]{Natterer_textbook} $\int_{\mathrm{S}^1}e^{-\mu x\cdot \theta^{\perp}+i(x\cdot \theta)t}d\theta= 2\pi J_0(\vert x\vert(t^2-\mu^2)^{1/2})$. Thus,
\begin{align*}
T^{\sharp}_{-\mu}K_{\rho}(x)&=\int\limits_{\vert\mu\vert<\vert t\vert <\sqrt{(1/\rho^2)+\mu^2}}\vert t \vert J_0(\vert x\vert(t^2-\mu^2)^{1/2})dt\\
&=2\int\limits_{0}^{1/\rho}\sigma J_0(\vert x\vert \sigma)d\sigma \quad \quad (\sigma=(t^2-\mu^2)^{1/2})\\
&=4\pi \big(\frac{1}{2\pi}\int\limits_{0}^{1/\rho}\sigma J_0(\vert x\vert \sigma)d\sigma\big)\\
&=4\pi \mathbf{\delta}^{1/\rho}(x) \quad \quad \cite[ \text{(1.3), Page 183}]{Natterer_textbook}
\end{align*}
where $$\displaystyle{\delta^{1/\rho}(x)=\frac{1}{2\pi}\int\limits_{\vert t\vert <1/\rho}e^{ix\cdot t}dt}=\frac{1}{2\pi}\int_{\mathbb{R}}I_{\rho}(t)e^{ix\cdot t}dt.$$ is an approximate Dirac-delta function that converges to Dirac distribution $\delta(x)$ pointwise (in the space of tempered distributions) as $\rho\to 0$. This completes the proof.
\end{proof}
\section{An {asymptotically unbiased} estimator for class H($\beta$,L)} \label{estimator}
In this section we propose a statistical estimator for $f\in H(\beta,L)$ based on the model (\ref{obs_model}) in the stochastic problem of exponential Radon transform. Inspired by the estimator proposed in \cite{Tsybakov_1991} and in Theorem \ref{fbp_formula} above, let us consider the statistical estimator:
\begin{align}
f_n^*(x)=\frac{1}{n}\sum_{i=1}^n e^{-\mu x\cdot \theta_i^{\perp}}K_{\rho_n}(\langle x\cdot \theta_i \rangle -s_i)Y_i
\end{align}

\noindent where $\theta_i, s_i$ and $Y_i$ are i.i.d. random variables as per the model (\ref{obs_model}) and $\rho_n \to 0$ as $n\to \infty$. We will call $\rho_n$ as the bandwidth of the estimator. Note that the MSE of the estimator in the non-parametric setting can be broken down in to two terms a ``bias term" and a ``variance term":
\begin{align}
\text{MSE}(f_n^*,f)&=E_f[(f_n^*(x)-f(x))^2]\nonumber \\
&=(E_f(f_n^*(x))-f(x))^2+E_f[(f_n^*(x)-E_f(f_n^*(x)))^2]\nonumber \\ &=B_n^2(x)+V_n^2(x).
\end{align}
where $B_n(x)$ is the bias of the estimator and $V_n^2(x)$ is its variance. 
Note that $$\text{MISE}(f_n^*,f)=\vert \vert B_n(x)\vert\vert^2_2+\vert \vert V_n(x)\vert\vert^2_2$$ where $\vert\vert (\cdot)\vert\vert_2$ denotes $L^2$ norm.
Recall that an estimator is said to be {asymptotically unbiased} if its bias goes to zero pointwise as the number of observations (samples) $n$ grows. We will now show that the estimator proposed above is asymptotically unbiased.
\begin{theorem}
	Let $(\theta_i,s_i)$, $i=\{1,\dots,n\}$ be i.i.d. random variables uniformly distributed on $Z=\mathrm{S}^1\times [-1,1]$ and these points be independent of the errors $(\epsilon_1,\dotsm\epsilon_n)$. If we consider the kernel estimator $f_n^*(x)=\frac{1}{ n}\sum_{i=1}^n e^{-\mu x\cdot \theta_i^{\perp}}K_{\rho_n}(\langle x\cdot \theta_i \rangle -s_i)Y_i$, then for each $ x \in B_1(x)$ the bias term, $B_n(x)=(E_f(f_n^*(x))-f(x))$, for this estimator  goes to zero as $n \to \infty$.
\end{theorem}

\begin{proof}
	It suffices to show that $E_f(f_n^{*}(x))=f_{\rho_n}(x)$ where $f_{\rho_n}(x)$ is given by Theorem (\ref{fbp_formula}). Then since $\rho_n\to 0$ as $n \to \infty$, hence $E_f(f_n^{*}(x))=f_{\rho_n} \to f(x)$ pointwise. In what follows, we will say that the i.i.d random variables $\theta_i$ have the same distribution as some random variable $\theta$, all $s_i$ are distributed with the same distribution as some random variable $s$ and similarly $Y$ and $\epsilon$ are random variables with the same distribution as random variables $Y_i$ and $\epsilon_i$ respectively. We will also denote by $E_{(\theta,s)}(\cdot)$ the expected value of a random variable with respect to the joint distribution of $(\theta,s)$ and by $E_{f|(\theta,s)}(\cdot)$ the conditional expectation of a random variable given $(\theta,s)$. Consider,
	\begin{align*}
	E_f(f_n^*(x))&=\frac{1}{n}E_f(\sum_{i=1}^n e^{-\mu x\cdot \theta_i^{\perp}}K_{\rho_n}(\langle x\cdot \theta_i \rangle -s_i)Y_i)\\
	&=E_f(e^{-\mu x\cdot \theta^{\perp}}K_{\rho_n}(\langle x\cdot \theta \rangle -s)Y)\\
	&=E_{(\theta,s)}\big(E_{f|(\theta,s)}(e^{-\mu x\cdot \theta^{\perp}}K_{\rho_n}(\langle x\cdot \theta \rangle -s)(T_{\mu}f(\theta,s)+\epsilon))\big) \quad \text{(law of iterated expectaion)}\\
	&=E_{(\theta,s)}\big(E_{f|(\theta,s)}(e^{-\mu x\cdot \theta^{\perp}}K_{\rho_n}(\langle x\cdot \theta \rangle -s)(T_{\mu}f(\theta,s))) \quad (\epsilon\text{ has mean  }0)\\
	&=E_{(\theta,s)}(e^{-\mu x\cdot \theta^{\perp}}K_{\rho_n}(\langle x\cdot \theta \rangle -s)(T_{\mu}f(\theta,s)))\\
	&=\frac{1}{4\pi}\int_{\mathrm{S}^1}e^{-\mu x\cdot \theta^{\perp}}\int_{-1}^{1}K_{\rho_n}(\langle x\cdot \theta \rangle -s)(T_{\mu}f(\theta,s))ds d\theta\\
	&=f_{\rho_n}(x)
	\end{align*}
\end{proof}
\section{Optimality of the estimator}
In this section we will show first of all that while the bias of the estimator decreases as bandwidth goes to zero, the variance increases as bandwidth decreases. Thus an optimal rate of convergence can be obtained by finding a suitable bandwidth $\rho_n$ which balances the bias and the variance term. Furthermore, we will establish the optimality of the proposed estimator under both semi-norms $d_1$ and $d_2$ as defined in Section 2. Let us now analyze the bias and the variance terms one by one. It is easy to check that for $\beta>1$ the following relations hold, 
\begin{align}
\vert I_{\rho_n}(t)-1\vert &\leq (\vert t\vert \rho_n)^{\beta}\label{relation1}\\
\vert I_{\rho_n}(t)-1\vert & \leq \bigg[\frac{2 \vert t\vert \rho_n}{1+\vert t\vert \rho_n} \bigg]^{\beta}\label{relation2}
\end{align} 
Consider first the bias term, $B_n(x)=f_{\rho_n}(x)-f(x)=\delta^{1/\rho_n}\star f(x) -f(x)$. Then for any fixed point $x\in B_1(x)$ and $\beta>1$:

\begin{align}
 B_n(x)&=\vert(\delta^{1/\rho_n}\star f(x_0)-f(x_0))\vert \nonumber \\
&\leq\frac{1}{2\pi}\int_{\mathbb{R}}\vert(I_{\rho_n}(\vert \xi\vert)-1)\vert \vert\tilde{f}(\xi)\vert d\xi \nonumber\\
&\leq \frac{1}{2\pi} \int_{\mathbb{R}} \vert \tilde{f}(\xi)\vert (2( \vert \xi \vert \rho_n))^{\beta}/(1+(\vert\xi\vert \rho_n)^{\beta}) d\xi \quad \quad\quad\quad \quad(\text{using } (\ref{relation2})) \nonumber\\
&=\frac{\rho_n^{\beta}}{\pi}[\int_{\mathbb{R}^2}\vert\tilde{f}(\xi)\vert^2 \vert \xi \vert ^{2\beta}d\xi]^{\frac{1}{2}} [\int_{\mathbb{R}^2} (1+(\vert\xi\vert \rho_n)^{\beta})^{-2}d\xi]^{\frac{1}{2}} \quad \quad (\text{using H\"{o}lder's inequality})\nonumber\\
&={c_1}{\rho_n^{\beta-1}}, \quad \quad \quad c_1>0
\end{align}

\noindent Anticipating the calculations required to show optimality using norm $d_2$, we also find an estimate for $\vert\vert B_n(x)\vert\vert_2^2$.
\begin{align}
\vert\vert B_n(x)\vert\vert_2^2&=\vert \vert \delta^{1/\rho_n}\star f(x) -f(x)\vert\vert_2^2\nonumber\\
&=\frac{1}{2\pi}\int_{\mathbb{R}^2} \vert (I_{\rho_n}(\vert \xi\vert)-1)\vert^2\vert \tilde{f}(\xi)\vert^2 d\xi \quad \quad (\text{using Parseval's theorem})\nonumber\\
&\leq \frac{1}{2\pi}\int_{\mathbb{R}^2}\vert\tilde{f}(\xi)\vert^2 (\vert \xi\vert\rho_n)^{2\beta}\quad \quad \quad \quad\quad \quad \quad \text{(using \ref{relation1})}\\
&\leq \frac{L\rho_n^{2\beta}}{2\pi}={c_2\rho_n^{2\beta}}
\end{align}
where $c_2=L/2\pi$.
Now we estimate the variance. 
\begin{Lemma}
	$V_n^2(x)=E_f\bigg((f_n^*(x)-E_f(f_n^*(x))^2)\bigg)\leq c_3/n\rho_n^3$ for $x \in B_1(x)$ and for some constant $c_3>0$. From this it also follows that for $x\in B_1(x)$, $\vert\vert V_n(x)\vert\vert_2^2\leq c_4/n\rho_n^3$ for some constant $c_4$.
\end{Lemma}
\begin{proof}
	In the following, $Var$ will denote the variance as per standard notation. First of all, note that $E_f(f_n^*(x))=f_{\rho_n}(x)$ and $s_i,\theta_i$ and $Y_i$ are i.i.d. random variables.  Thus,
	\begin{align*}
	V_n^2(x)&=E_f\bigg((f_n^*(x)-E_f(f_n^*(x)))^2\bigg)\nonumber \\
	&=\frac{1}{n}\bigg(Var_{f}(e^{-\mu(x\cdot\theta^\perp)}K_{\rho_n}(x\cdot\theta-s)T_{\mu}f(\theta,s)) \bigg)+\frac{1}{n}\bigg(E_f(e^{-\mu(x\cdot\theta^\perp)}K_{\rho_n}(x\cdot\theta-s)\epsilon^2)\bigg) \\
	&\leq\frac{\sigma^2+4e^{\lvert \mu \rvert}L^2}{4\pi n}\int_{\mathrm{S}^1}e^{-2\mu (x\cdot\theta^{\perp})}\int\limits_{-1}^{1}K^2_{\rho_n}(x\cdot\theta-s)dsd\theta
	\end{align*}
	where we use the fact that since $f\in H(\beta,L)$ is compactly supported in $B_1(x)$, we get $\lvert T_{\mu}f(\theta,s)\rvert \leq 2e^{|\mu|}L.$
	Let us now estimate:
	\begin{align}
	\int\limits_{-1}^{1}K^2_{\rho_n}(x\cdot\theta-s)ds &\leq \int\limits_{-\infty}^{\infty}\vert K_{\rho_n}(s)\vert ^2ds \nonumber\\
	 &\leq\int\limits_{-\infty}^{\infty}\vert \widetilde{K}_{\rho_n}(s)\vert ^2ds \quad \quad \quad \quad \text{(using Parseval's theorem)}\nonumber \\
	&= \frac{1}{3}\bigg[({(1/\rho_n^2)+\mu^2})^{3/2}-\vert\mu\vert^3\bigg] \nonumber \\
	&=\frac{1}{3} \bigg[\big(({(1/\rho_n^2)+\mu^2})^{1/2}-\vert\mu\vert\big)\big((1/\rho_n^2)+2\mu^2+\vert\mu\vert((1/\rho_n^2)+\mu^2)^{\frac{1}{2}}\big)\bigg] \nonumber \\
	&=\frac{1}{3}\bigg[\frac{(1/\rho_n^2)\big((1/\rho_n^2)+2\mu^2+\vert\mu\vert((1/\rho_n^2)+\mu^2)^{\frac{1}{2}}\big)}{\big(({(1/\rho_n^2)+\mu^2})^{1/2}+\vert\mu\vert\big)}\bigg]\nonumber \\
	&\leq \frac{(3+\sqrt{2})(1/\rho_n^3)}{3} \nonumber
	\end{align}
	where we have used the fact that we choose $\vert \mu \vert \leq (1/\rho_n)$.
\noindent Thus 
\begin{align}
	V_n^2(x)& \leq \frac{(3+\sqrt{2})(\sigma^2+4e^{|\mu|}L^2)}{4 \pi n \rho_n^3}\int_{\mathrm{S}^1}e^{-2\mu x\cdot \theta^{\perp}}d\theta \nonumber \\
	&\leq \frac{c_3}{n\rho_n^3}\quad \quad \quad \quad \quad (\text{for }x \text{ in }B_1(x) )
\end{align}
where $c_3>0$ is a constant. Now $\vert\vert V_n(x)\vert\vert_2^2=\int_{x\in B_1(x)}V_n^2(x)dx\leq c_4/n\rho_n^3$ for some constant $c_4$.

\end{proof}
\begin{theorem}\label{ub_thm2}
Let $f\in H(\beta,L)$ where $\beta>1$ and $f_n^*(x)$ be the estimator defined in section \ref{estimator}. Let $\theta_i,s_i$ for $i=1.\dots, n$ be i.i.d.\ random variables and the observation model corresponding to the problem of ERT be given by (\ref{obs_model}). Let $x_0\in B_1(x)$ be some fixed point. In the the definition of risk in section \ref{definitions} let us use the seminorm $d_1(f,g)=\vert f(x_0)-g(x_0)\vert$ where $x_0\in B_1(x)$ is some arbitrary point. Let $\rho_n=\alpha_1n^{-1/(2\beta+1)}$ for some constant $\alpha_1$, then the following upper bound holds:
	$$\sup_{f\in H(\beta,L)}\psi_n^{-2}\text{MSE}(f_n^*,f)\leq C_0$$
	where $\psi_n=n^{-\frac{\beta-1}{2\beta+1}}$.
\end{theorem}
\begin{proof}
	\begin{align*}
	\text{MSE}(f_n^*,f)&=B_n^2(x_0)+V_n^2(x_0)\\
	&\leq c_1^2\rho_n^{2\beta-2}+\frac{c_3}{n\rho_n^3}.
	\end{align*}
	The minimum of the RHS is obtained for $\rho_n^*=(\frac{3c_3}{2c_1^2(\beta-1)})^{\frac{1}{2\beta+1}}[n^{-\frac{-1}{2\beta+1}}]$. With this choice of $\rho_n=\rho_n^*$, MSE$(f_n^*,f)=\mathcal{O}(n^{-(2\beta-2)/(2\beta+1)})$. 
\end{proof}
\begin{theorem} \label{ub_thm1}
Let $f\in H(\beta,L)$ where $\beta>1$ and $f_n^*(x)$ be the estimator defined in section \ref{estimator}. Let $\theta_i,s_i$ for $i=1.\dots, n$ be i.i.d.\ random variables and the observation model corresponding to the problem of ERT be given by (\ref{obs_model}). Consider the seminorm given by $d_2(f,g)=\vert\vert f-g\vert\vert_2$ where $\vert \vert (\cdot)\vert\vert_2$ indicates the $L_2$ norm as usual. Let $\rho_n=\alpha_2 n^{-1/(2\beta+3)}$, where $\alpha_2 $ is a constant. Then the following upper bound holds,
$$\sup_{f\in H(\beta,L)}\Psi_n^{-2}\text{MISE}(f_n^*,f)\leq C_1$$
where $\Psi_n=n^{-\beta/(2\beta+3)}$ and a positive constant $C_1$.
\end{theorem}
\begin{proof}
\begin{align*}
\text{MISE}(f_n^*,f)&=\vert \vert B_n(x)\vert\vert_2^2+\vert \vert V_n(x)\vert\vert_2^2\\ 
&\leq c_2\rho_n^{2\beta}+c_4/n\rho_n^3.
\end{align*}
Note that the minimum of the RHS above is attained for $\rho_n^*=(\frac{3c_4}{2c_2\beta})^{\frac{1}{2\beta+3}}[n^{\frac{-1}{2\beta+3}}]$. With this choice of $\rho_n=\rho_n^*$, $\text{MISE}(f_n^*,f)=\mathcal{O}(n^{-2\beta/(2\beta+3)})$. This completes our proof.
\end{proof}

\noindent The upper bounds established in Theorems  \ref{ub_thm2} and {\ref{ub_thm1}}  above imply that the minimax risks for the estimator using the two seminorms $d_2$ and $d_1$ is bounded above by $C_1\Psi_n^2$ and $C_2\psi_n^2$ respectively where $\Psi_n$ and $\psi_n$ are sequences that go to zero as $n \to \infty$ . As per Definition (\ref{optimality}), to establish the optimality of the estimator we need to show that each of the two minimiax risks also satisfy the corresponding lower bounds. To that end, at first we make the following additional assumptions for the observation model \ref{obs_model}:\\
\textbf{Assumption on the distribution of noise (B1):} The random variables $\epsilon_i$ are i.i.d having a distribution $G(\cdot)$ that satisfies :
\begin{align}\label{noise_assum}
\int\limits_{-\infty}^{\infty}\ln\frac{dG(u)}{dG(u+v)}dG(u)\leq I_0v^2, \quad \quad \vert v\vert \leq v_0
\end{align}
where $I_0>0$ and $v_0>0$ are some constants. \\
\noindent \textbf{Assumption on design points (B2)}: Any design, i.e. $\{\theta_i,s_i\}_{i=1}^n$ on the cylinder $Z=\mathrm{S}^1\times [-1,1]$ will be said to be feasible if any non-negative measurable function $g(\theta,s)$ defined on $Z$ satisfies:
\begin{align}\label{Design_assum}
E_{(\theta,s)}\bigg[\sum\limits_{i=1}^{n}g(\theta_i,s_i)\bigg]\leq C_3\int\limits_{Z}g(\theta,s)dsd\theta.
\end{align}
In what follows, we will assume that the design is feasible in the sense described above.
\begin{theorem} \label{point_optimality} Let $\beta, f, f_n^*, \theta_i,s_i$ as in Theorem \ref{ub_thm2}. If in addition, assumptions B1 and B2 are satisfied by the observation model (\ref{obs_model}) then the following inequality holds:
$$\liminf_{n\to \infty}\quad\inf_{\hat{f}_n}\quad\sup_{f\in H(\beta,L)}\psi_n^{-2} \text{MSE}(\hat{f}_n,f)\geq c_0$$
where $\psi_n$ is the same sequence as in Theorem {\ref{ub_thm2}}, $\displaystyle{\inf_{\hat{f}_n}}$ denotes the infimum over all estimators and $c_0>0$ is some constant.
\end{theorem}
\begin{proof}
The proof method follows that in \cite[Theorem 4]{Tsybakov_1991} and we will adapt their proof wherever needed. As noted there, using standard reduction techniques for establishing lower bounds on the minimax risk of regression estimators in a non-parametric setting, the problem can be reduced to {showing that the Kullback distance between the two probability measures corresponding to two appropriately chosen functions (hypothesis) is bounded}, see also \cite[section 2.5]{Tsybakov_book}. Thus consider the functions {(hypothesis)} $f_0(x)=0$ and $f_1(x)=Ah^{\beta-1}\eta_0((x-x_0)/h)$ where $h=n^{-\frac{1}{2\beta+1}}$, $\eta_0(x)\in H(\beta,L)$ is a compactly supported bounded function such that $\eta_0(0)>0$ and $0<A<1$ is a constant. Following \cite{Tsybakov_1991}, we will first show that $f_1(x)\in H(\beta,L)$. Note that:
\begin{align*}
\tilde{f}_1(\xi)&=Ah^{\beta-1}\int \eta_0((x-x_0)/h)e^{i\xi\cdot x}dx=Ah^{\beta-1}e^{i\xi\cdot x_0}\int\limits_{\mathrm{S}^1}\int\limits_{0}^{\infty}(u)\eta_0(u\theta/h)e^{i\xi\cdot u\theta} dud\theta\\
&=Ah^{\beta+1}e^{i\xi\cdot x_0}\int\limits_{\mathrm{S}^1}\int\limits_{0}^{\infty}(\bar{u})\eta_0(\bar{u}\theta)e^{i(h\xi\cdot \theta)\bar{u}}d\bar{u}d\theta=Ah^{\beta+1}e^{i\xi\cdot x_0}\tilde{\eta}_{0}(h\xi).
\end{align*}
Thus,
\begin{align*}
\int (1+\vert \xi^2\vert)^{\beta}\vert \tilde{f}_1(\xi)\vert^2d\xi&=A^2h^{2(\beta+1)}\int (1+\vert \xi\vert^2)^{\beta}\vert \tilde{\eta}_0(h\xi)\vert^2d\xi\\&=A^2\int(h^2+\vert \bar{\xi}\vert ^2)^{\beta}\vert \tilde{\eta}_{0}(\bar{\xi})\vert^2d\bar{\xi}\leq L
\end{align*}
where we have used the fact that $0<h,A<1$ and $\eta_0(x)\in H(\beta,L)$. Also observe that $\vert f_1(x_0)-f_0(x_0)\vert=Ah^{\beta-1}\eta_0(0)$ and $\eta_0(0)>0$ by assumption. Now let $P_0$ and $P_1$ be probability measures corresponding to the experiments with observations given by the regression model (\ref{obs_model}) for $f=f_0$ and $f=f_1$ respectively and {$p_0$ and $p_1$ be the densities corresponding to the measures $P_0$ and $P_1$ respectively}. Then to complete the proof of the theorem it suffices to show the Kullback information distance between the two measures, $I(P_0,P_1)\leq 1/2$. Again, From \cite{Tsybakov_1991},
{
\begin{align}{\label{estimate1}}
I(P_0,P_1)&=\int \ln \bigg(\frac{dP_0}{dP_1}\bigg)dP_0 = E_{f_0}\int \ln\bigg(\frac{dp_0}{dp_1}\bigg)d\nu \quad \quad (\nu \text{ is the Lebesgue measure })\nonumber\\
 &=E_{(\theta,s)}\bigg(E_{f_0\vert (\theta,s)}\int \ln\bigg(\frac{dp_0}{dp_1}\bigg)d\nu\bigg)\nonumber \\&=E_{(\theta,s)}\bigg[\sum\limits_{i=1}^n\int \ln \frac{dG(v-T_{\mu}f_0(\theta_i,s_i))}{dG(v-T_{\mu}f_1(\theta_i,s_i))} dG(v-T_\mu f_0(\theta_i,s_i))\bigg]\nonumber \quad \quad (\text{ see }\cite[(2.36)]{Tsybakov_book})\\&\leq C_3nI_0\int\limits_Z\vert T_{\mu}f_1(\theta,s)\vert^2 dsd\theta \quad \quad \text{(using B1 and B2)}
\end{align}}
To estimate $\int\limits_Z\vert T_{\mu}f_1(\theta,s)\vert^2 dsd\theta $, we will follow \cite[section 4]{Rigaud2015}.
Consider a function $\phi(x)\in \mathcal{S}(\mathbb{R}^2)$ (i.e. Schwartz class) such that $\phi(x)=1$ for $x\in B_1(x)$. Let us introduce \begin{align}\label{wdef}
\bar{w}(x,\theta)=\phi(x)e^{\mu x\cdot\theta^{\perp}}
\end{align}
Clearly for any function $f_1(x)$ supported in $B_1(x)$,
\begin{align*}
T_{\mu}f_1(\theta,s)=T_{\bar{w}}f_1(\theta,s)=\int\limits_{\mathbb{R}^2}\bar{w}(x,\theta)f(x)\delta(x\cdot\theta-s)dx
\end{align*} 
Taking the Fourier transform of $T_{\bar{w}}f_1(\theta,s)$ with respect to the $s$- variable we get the following inequality \cite[equation 27]{Rigaud2015},
\begin{align}
\vert\tilde{T}_{\bar{w}}{f(\theta,t)}\vert^2 \leq (2\pi)^{-1}\vert W_{\bar{w}}\star \tilde{f}(\xi)\vert^2
\end{align}
where $W_{\bar{w}}=\sup\limits_{\theta\in \mathrm{S}^1}\vert \tilde{\bar{w}}(\theta,t)\vert$ and $\tilde{(\cdot)}$ indicates the corresponding Fourier transform (either $1$-d or $2$-d) as usual.
Now from \cite[equation 29]{Rigaud2015},
\begin{align}\label{estimate2}
\vert\vert T_{\mu}f_1(\theta,s)\vert\vert^2_{L^2(Z)} \leq \vert\vert T_{\mu}f_1(\theta,s)\vert\vert^2_{H^{1/2}(Z)}\leq K\vert\vert W_{\bar{w}}\vert \vert^2 _{L^{1}(\mathbb{R}^2)}\vert\vert f_1\vert\vert^2_{L^2({\mathbb{R}^2})}=\bar{K}\vert\vert {f}_1\vert\vert^2_{L^2{(\mathbb{R}^2)}}
\end{align}
where $\bar{K}=K\vert\vert W_{\bar{w}}\vert\vert_{L^1(\mathbb{R}^2)}$. We note in passing that since $\bar{w}(x,\theta)$ is given by (\ref{wdef}), $\vert\vert W_{\bar{w}}\vert\vert_{L^1(\mathbb{R}^2)}$ is finite.

\noindent Now $\vert\vert {f}_1\vert\vert^2_{L^2{(\mathbb{R}^2)}}=A^2h^{2\beta-2}\int\limits_{\mathbb{R}^2}\vert\eta_0((x-x_0)/h)\vert^2 dx=A^2h^{2\beta+1}\int\limits_{\mathbb{R}^2}\vert\eta_0(y)\vert^2 dy$. Since $\eta_0\in H(\beta,L)$ is compactly supported bounded function, thus $\vert\vert \eta_0(y)\vert\vert^2_{2}$ is finite.
Thus,
\begin{align}
I(P_0,P_1)\leq C_3I_0\bar{K}A^2\vert\vert \eta_0(y)\vert\vert_2^2nh^{2\beta+1}=C_3I_0\bar{K}A^2\vert\vert \eta_0(y)\vert\vert_2^2 \quad \quad (h=n^{-\frac{1}{2\beta+1}})
\end{align}
Thus if we choose $A$ to be small enough, $I(P_0,P_1)\leq 1/2$.
\end{proof}
\begin{remark}
Note that Theorems \ref{ub_thm2} and \ref{point_optimality} together establish the optimality of the convergence rate of minimax risk for the estimator proposed in Section \ref{estimator} under the seminorm $d_1$.
\end{remark}
\begin{theorem}\label{opti_l2}
Let $\beta, f, f_n^*, \theta_i,s_i$ as in Theorem \ref{ub_thm1}. If in addition, assumptions B1 and B2 are satisfied by the observation model (\ref{obs_model}) then the following inequality holds:
$$\liminf_{n\to \infty}\quad\inf_{\hat{f}_n}\quad\sup_{f\in H(\beta,L)}\Psi_n^{-2} \text{MISE}(\hat{f}_n,f)\geq c_1$$
where $\Psi_n$ is the same sequence as in Theorem {\ref{ub_thm1}}, $\displaystyle{\inf_{\hat{f}_n}}$ denotes the infimum over all estimators and $c_1>0$ is some constant.
\end{theorem}
\begin{proof}
First of all, we recall from \cite[section 2.6]{Tsybakov_book} that to establish lower bounds for the convergence rate of the estimators in $L_p$ seminorms requires us to work with many hypotheses (M-hypotheses) instead of just two as we did in the proof of Theorem \ref{point_optimality} above. The proof of this theorem follows that of \cite[Theorem 5]{Tsybakov_1991}. All the geometric arguments in this proof are identical to the geometrical arguments in \cite{Tsybakov_1991} and we only need to change the argument wherever an estimate for the usual Radon transform is to be replaced with an analogous estimate for the exponential Radon transform. For the sake of completeness, we outline the proof given in \cite{Tsybakov_1991} here, adapting it to the case of ERT wherever needed.\\
\noindent Consider a collection of non-intersecting balls $\Delta_k, k\in \{1,\dots,M\}$ inscribed in $B_1(x)$ with center $a_k$ and of radius $1/m$ such that $m$ and $M$ are sequences and $m\to \infty $ as $n\to \infty$. Furthermore, one can choose $m$ and $M$ (the precise choice for $m$ is described later) such that the following relation is satisfied:
\begin{align}
C_4 m^2\leq M\leq C_5 m^2
\end{align} 
Let $\eta(x)$ be a smooth function supported in $B_1(x)$. Then each function $\eta_k(x)=\eta(m(x-a_k))$ is supported respectively in $\Delta_k$. To each $m$-tuple $b=(b_1,\dots,b_m)$ where $b_k$ is either $0$ or $1$, we associate a function $f(x,b)$ supported in $B_1(x)$ such that:
\begin{align*}
f(x,b)=Am^{-\beta}\sum\limits_{k=1}^{M}b_k\eta_k(x)
\end{align*}
where $A>0$ will be chosen in a manner described below. 
We state without proof the following two lemmas from \cite{Tsybakov_1991}:
\begin{Lemma} \cite[Lemma 3]{Tsybakov_1991}
There exists $A_{\beta}>0$ such that for $A<A_{\beta}$, the function $f(x,b)\in H(\beta,L)$ for any $m$-tuple $b$.	
\end{Lemma}
\noindent Consider any design $\mathcal{D}_n=\{(\theta_i,s_i)\}_{i=1}^{i=n}$, and consider the lines $L_i=\{x\in \mathbb{R}^2:x\cdot\theta_i=s_i\}$. Take the set of balls $\Delta_k$ such that each ball intersects at most $C_6n/m$ lines where $C_6>0$ is a constant, whose choice is described in Lemma \ref{geometric_lemma} below. Let the set of indices $J$ be defined as:
\begin{align*}
J=J(\mathcal{D}_n)=\{k\in \{1,\dots,M\}:& \text{ number of lines corresponding to }\mathcal{D}_n\text{ that intersect with }\Delta_k\\&\text{ is less than or equal to }C_6n/m \}
\end{align*}

\begin{Lemma}\cite[Lemma 4]{Tsybakov_1991} \label{geometric_lemma}
There exists $C_6>0$ such that for any design $\mathcal{D}_n$, we have the inequality:$$\mathrm{card}J>M/2.$$
\end{Lemma}
\noindent In what follows, $C_6$ is chosen such that Lemma \ref{geometric_lemma} is satisfied.
\end{proof}
\noindent Following \cite{Tsybakov_1991}, 
 let us also indicate by, $b^{(k,0)}=\{b_1,\dots,b_{k-1},0,b_{k+1},\dots,b_M\}$ and $b^{(k,1)}=\{b_1,\dots,b_{k-1},0,b_{k+1},\dots,b_M\}$ $M$-tuples with fixed $k$-th elements as indicated. Furthemore, we use the following notation for functions:
$$f_{k_0}=f(x,b^{(k,0)}) \quad \text{and}\quad f_{k_1}=f(x,b^{(k,1)}). $$ Let $g_k(x)=f_{k_0}(x)-f_{k_1}(x)$ which is supported only on $\Delta_k$ by construction. 
Let $P_{k_0}$ and $P_{k_1}$ be the probability measures corresponding to the model \ref{obs_model} for $f=f_{k_0}$ and $f=f_{k_1}$. Let $I(P_{k_0}, P_{k_1})$ be the Kullback information distance between these two probability measures. Thus from \cite{Tsybakov_1991}, the desired lower bound for the minimax rate will be obtained if we can show that for a sufficiently small $C_8>0$ such that $m=(C_8n)^{\frac{1}{2\beta+3}}$, $I(P_{k_0}, P_{k_1})<1/2$.
Just as in \cite{Tsybakov_1991} and similar to the proof of Theorem \ref{point_optimality} above, from assumptions B1 and B2, we get:
\begin{align}\label{KB_dist_ineq}
I(P_{k_0},P_{k_1})\leq I_0 \sum_{i=1}^{n}(T_{\mu}g_k(\theta_i,s_i))^2
\end{align}
Now from the definition of ERT and from the fact that $\eta_k(x)$ is supported in $\Delta_k\subset B_1(x)$,
\begin{align}
\vert (T_\mu g_k)(\theta_i,s_i)\vert &=\bigg\vert \int\limits_{L_i\cap \Delta_k} e^{\mu x\cdot \theta_i^{\perp}} A m^{-\beta}\eta(m(x-a_k))dx\bigg\vert\nonumber \\
&\leq {C_9} \int\limits_{L_i\cap \Delta_k} \vert A m^{-\beta}\eta(m(x-a_k))\vert dx \quad \quad (C_9=\sup_{x\in B_1(x)}e^{\mu x\cdot \theta^{\perp}})\nonumber\\
&\leq C_{10} m^{-\beta-1}
\end{align}
Now note that since $k\in J$, thus at most $C_6n/m$ of the terms in the sum on RHS of (\ref{KB_dist_ineq}) are non zero.
Putting it all together, we have :
\begin{align}
I(P_{k_0},P_{k_1})\leq I_0 C_6 (C_{10})^2 (n/m) m^{-2\beta-2}\leq I_0 C_6 C_{10}^2 C_8.
\end{align}
Thus if we choose $C_8\leq \frac{I_0 C_6 C_{10}^2}{2}$, then we get $I(P_{k_0},P_{k_1})\leq 1/2$ as desired. This completes the proof of the theorem.
\begin{remark}
	Note that Theorems \ref{ub_thm1} and \ref{opti_l2} together establish the optimality of the estimator in the $d_2$ semi-norm setting.
\end{remark}

\bibliographystyle{plain}
\bibliography{Refs2stat}

\end{document}